\documentclass[11pt,leqno]{article}
\usepackage[latin1]{inputenc}
\usepackage{amsmath,amssymb,latexsym}
\usepackage[pdftex]{graphicx}
\author{Klaus Pommerening\\
        Johannes-Gutenberg-Universit\"at\\
        Mainz, Germany\\
        {\tt pommeren@uni-mainz.de}}
\title{The Indecomposable Solutions of Linear Congruences}
\date{}

\newtheorem{theorem}{Theorem}
\newtheorem{prop}{Proposition}
\newtheorem{lemma}{Lemma}
\newtheorem{kor}{Corollary}

\newcommand*{\N}{\mathbb{N}}

\newcommand*{\Z}{\mathbb{Z}}

\newcommand*{\supp}{\operatorname{supp}}
\newcommand*{\charac}{\operatorname{char}}
\newcommand*{\Oh}{\operatorname{O}}

\newcounter{plcnt}
\newenvironment*{proplist}%
{\begin{list}{\rm (\roman{plcnt})}%
{\usecounter{plcnt} \setlength{\labelwidth}{7mm}}}%
{\end{list}}

\newenvironment*{proof}{{\it Proof.}}{$\Diamond$ \par \vspace{3ex}}

\begin{document}
\maketitle

\begin{description}
\item[Abstract]
   This article considers the minimal non-zero (= indecomposable) solutions of the
   linear congruence $1\cdot x_1 + \cdots + (m-1)\cdot x_{m-1} \equiv 0 \pmod m$
   for unknown non-negative integers $x_1, \ldots, x_n$, and characterizes
   the solutions that attain the Eggleton-Erd\H{o}s bound. Furthermore it
   discusses the asymptotic behaviour of the number of indecomposable solutions.
   The results have direct interpretations in terms of zero-sum sequences and
   invariant theory.
\end{description}

\noindent A typical problem of additive number theory is the linear congruence:
Given $m \in \N_2$ and $a \in \Z^n$, determine $x \in \N^n$ with
\[
  {\bf (A)}\qquad\qquad\qquad\qquad
    a_1 x_1 + \cdots + a_n x_n \equiv 0 \pmod m.
  \qquad\qquad\qquad\qquad
\]
(Without loss of generality $0 \leq a_i < m$ for all $i$.)
\begin{quote}
   Note that in this article $\N$ stands for the numbers $\{0, 1, 2, \ldots\}$,
   and $\N_k$ for $\{k, k+1, \ldots\}$. Think of $0$ as being the most natural
   number.
\end{quote}
In general it's trivial to find lots of single solutions \cite{MZS}.
But getting an overview over the complete solution set seems difficult,
in particular estimating the numbers of indecomposable solutions.

The linear congruence {\bf (A)} is easily reduced to the standard congruence
\[
  {\bf (C}_m{\bf )}\qquad\qquad\qquad\qquad
     x_1 + \cdots + (m-1) \cdot x_{m-1}  \equiv  0   \pmod m.
  \qquad\qquad\qquad\qquad
\]
For $m \leq 38$ the sequence {\tt A096337} of OEIS \cite{OEIS} indicates
the number of indecomposable solutions of ${\bf (C}_m{\bf )}$. In \cite{FMPT}
these numbers ($+1$) are even listed for $m$ up to $60$.
The paper \cite{DEN} gives a weak asymptotic lower bound.

This article characterizes the indecomposable solutions that attain the bound
found by Eggleton and Erd\H{o}s \cite{E-E}. Moreover it discusses the growth
of the number of indecomposable solutions as a function of $m$.

The results have direct applications to invariant theory, my motivation
to consider them, see \cite{DEN, Kac}. Another application domain is the theory
of zero-sums, see \cite{ChSm, Gao, Hami, YL, ZYL}, that is essentially
another view at the same mathematical subject.

\section{Indecomposable Solutions}\label{s:indsol}

The solution set of {\bf (A)} is the kernel of a homomorphism, hence a finitely
generated sub-monoid $H \leq \N^n$ by Dickson's lemma \cite{Dick}.
The canonical minimal system of generators consists of the indecomposable
(or irreducible, or minimal nonzero) solutions.
Thus solving the linear congruence {\bf (A)} or ${\bf (C}_m{\bf )}$ boils down
to determining the indecomposable solutions. Meaningful partial tasks are:
\begin{description}
   \item[(I)]   {\em Find bounds for the coordinates of the indecomposable
      solutions that are as strong as possible.}
   \item[(II)]  {\em Identify and characterize indecomposable solutions with special
	    properties.}
   \item[(III)]  {\em Find algorithms that construct all indecomposable
      solutions, and analyze their efficiency.}
   \item[(IV)] {\em Determine the number of indecomposable solutions, at least
      give good estimates of this number.}
\end{description}
We expect an exponential dependency of the number of indecomposable solutions
from $m$. In particular
an algorithm as in (III) must have exponential complexity and cannot be efficient
in the proper sense.

The case $n = 1$ of the linear congruence is trivial. Here is the result:

\begin{prop}\label{p:n=1}
   Let $m \in \N_2$ and $a \in \N_1$. Then the only indecomposable solution of
   the congruence $a x \equiv 0 \pmod m$ is the minimal integer $x > 0$ with
   $m | ax$. If $m$ and $a$ are coprime, $x = m$.
\end{prop}

The results for the case $n = 2$ are considerably more complex but known,
see \cite{Tin1} or \cite{KP_LC2}.

A naive algorithm for finding the indecomposable solutions $x \in \N^n$ of the
linear congruence starts with a finite subset $\mathcal{D} \subseteq \N^n$
that is guaranteed to contain all indecomposable solutions, checks the vectors
in $\mathcal{D}$ whether they solve the congruence, and reduces the list of
solutions to the indecomposable elements. The number of integer points in $\mathcal{D}$
is a coarse upper bound, the number of special solutions as in (II), a coarse
lower bound for the number of indecomposable solutions.

Classical results provide bounds for the coordinates of indecomposable solutions
that improve the trivial bound $x_i \leq m$: Let $x \in \N^{m-1}$ be an indecomposable
solution of ${\bf (C}_m{\bf )}$. Then
\begin{itemize}
   \item $\|x\|_1 \leq m$ (Tinsley \cite{Tin2}, a special case of Noether's
      bound \cite{Noe}),
   \item $\|x\|_1 + \sigma(x) \leq m+1$ (Eggleton/Erd\H{o}s \cite{E-E}),
   \item $\sigma(x) \leq \lceil 3\,\sqrt{m} \rceil$ (Olson \cite{Ol75}), see 
      the note~3 below,
\end{itemize}
where for a vector $x \in \N^n$ we denote by
\[
      \sigma(x) :=  \#\supp(x)\quad(\text{where }
      \supp(x) :=  \{i = 1, \ldots,n \:|\: x_i \neq 0\})
\]
the cardinality of its support, called the {\bf width} of $x$. The
Noether-Tinsley bound follows from Eggleton-Erd\H{o}s's since $\sigma(x) \geq 1$.

Moreover we call
\begin{itemize}
   \item $\|x\|_1 = x_1 + \cdots + x_n$ the {\bf length} (or degree \cite{HaWe}),
   \item $\|x\|_{\infty} = \max\{x_1, \ldots, x_n\}$ the {\bf height},
   \item $\|x\|_1 + \sigma(x)$ the {\bf total size} (= length + width),
   \item $\alpha(x) := x_1 + \cdots + n \cdot x_n$ the {\bf weight}
\end{itemize}
of $x$. Clearly in $\N$
\[
    \sigma(x) = \sum_{x_i \neq 0} 1 \leq \sum_{x_i \neq 0} x_i = \|x\|_1
    \leq \sum_{x_i \neq 0} i \cdot x_i = \alpha(x).
\]
The canonical unit vectors in $\N^n$ (or $\Z^n$) are $e_1 = (1,0,\ldots,0)$,
\ldots, \mbox{$e_n = (0,\ldots,0,1)$}.

\begin{description}
   \item[Remark] Assume an indecomposable solution $x = (x_1, \ldots, x_{m-1})$ of
      ${\bf (C}_m{\bf )}$ has a pair of coordinates $x_i > 0$ and $x_{m-i} > 0$
      with $i < m/2$. Then the solution $e_i + e_{m-i}$ is $\leq x$, hence $= x$.
      Therefore the width of an indecomposable solution $x$ is bounded by
      $\sigma(x) \leq \frac{m}{2}$, except for $m = 3$ and $x = (1,1)$.
      The Olson bound $\lceil 3\,\sqrt{m} \rceil$ is smaller than $\frac{m}{2}$
      only for $m > 36$.
\end{description}

The strong Davenport constant of an abelian group $M$, see \cite{ChFS}, is defined as
the maximum number of {\em different} elements in a minimal zerosum
multiset in $M$. (Remember that the Davenport constant is the maximum number of
not necessarily different elements in a minimal zerosum multiset.) This maximum
is attained by a minimal zerosum {\em set} (that is, without repeated elements), see \cite{ChFS}.

As a special case the strong Davenport constant of $\Z/m\Z$ is the largest
width of an indecomposable solution of ${\bf (C}_m{\bf )}$:
\[
     \text{\sf SD}(m) := \max\{\sigma(x) \:|\: x \text{ indecomposable solution of }
        {\bf (C}_m{\bf )}\},
\]
and there is an indecomposable solution $x$ of height $\|x\|_{\infty} = 1$ that attains
this bound. Thus for determining
$\text{\sf SD}(m)$ we need to consider only indecomposable solutions with all coordinates
equal to $0$ or $1$. Explicit values, easily determined by a simple program, are
\[
     \text{\sf SD}(m) = \begin{cases}
                           2 & \text{for } m = 3,\: 4,\: 5, \\
                           3 & \text{for } m = 6,\: 7, \\
                           4 & \text{for } m = 8,\: 9,\: 10, \\
                           5 & \text{for } m = 11,\: \ldots,\: 15, \\
                           6 & \text{for } m = 16,\: \ldots,\: 23.
                        \end{cases}
\]

\begin{description}
   \item[Notes] on the Erd\H{o}s-Heilbronn conjecture (EHC):
      \begin{enumerate}
         \item The EHC claims that a subset $S$ of an abelian group $M$ has a nontrivial
            subsum equal to $0$ if $s = \# S \geq c\,\sqrt{m}$ with $m = \#M$ for an absolute
            constant $c$. Erd\H{o}s and Heilbronn proved this for the cyclic group $M = \Z/p\Z$
            of prime order $p$ with $c = 3\,\sqrt{6}$. Olson \cite{Ol68} dropped
            the constant to $c = 2$ for prime order $p$, and \cite{Ol75} to $c = 3$ for
            arbitrary (even non-abelian) $M$.

         \item Let $c$ be the E-H constant valid for the abelian group $M$. Let $T \subseteq M$
            be a minimal zerosum set. Then $\# T \leq \left\lceil c \sqrt{m} \right\rceil$.

            For if $\# T > \left\lceil c \sqrt{m} \right\rceil$, then
            $\# T \geq \left\lceil c \sqrt{m} \right\rceil + 1$.
            Dropping an arbitrary element from $T$ results in a proper subset $S \subset T$
            of size $\# S \geq \left\lceil c \sqrt{m} \right\rceil$, hence containing a nontrivial
            zerosum subset. Therefore $T$ is not minimal.

         \item Olson's result \cite[Theorem 3.2]{Ol75}, applied to a subset $S \subseteq \Z/m\Z$
            of size $s$ with at most $s^2/9$ different subset sums, implies that $0$ is
            a nontrivial subset sum of $S$. The precondition on $s$ is obviously satisfied
            if $s^2/9 \geq m$, that is, $s \geq 3\,\sqrt{m}$. This yields Olson's bound.

         \item The strong form of the EHC (by Erd\H{o}s) drops the
            constant to \mbox{$c = \sqrt{2}$}. In this strong form the conjecture is open,
            the best known bound is $\sqrt{2p} + 5\,\log p$ for $m = p$ prime, and
            \mbox{$c = \sqrt{2m} + \varepsilon(m)$} where $\varepsilon(m)$ is
            $\Oh(\sqrt[3]{m}\cdot \log(m))$ for $G$ cyclic of order $m$,
            proved by Hamidoune and Z\'{e}mor \cite{HaZe}.

            Therefore we have
            \begin{itemize}
               \item $\text{\sf SD}(m) \leq \left\lceil 3\,\sqrt{m} \right\rceil$ (proved by Olson), and
               \item $\text{\sf SD}(m) \leq \left\lceil \sqrt{2m} \right\rceil$ (conjectured by Erd\H{o}s).
            \end{itemize}
            The explicit values above show that the bound $\left\lceil \sqrt{2m} \right\rceil$
            is sharp for many values of $m$.
      \end{enumerate}

   \item[Definition] Call a solution $x$ of ${\bf (C}_m{\bf )}$ {\bf extremal}
      if it is indecomposable and attains the Eggleton-Erd\H{o}s bound, that is,
      has the maximum possible total size $\|x\|_1 + \sigma(x) = m+1$.

   \item[Example 1] If $x$ is extremal and $\sigma(x) = 1$, then $\|x\|_1 = m$,
      thus $x = m\,e_i$ where $i$ is coprime with $m$, see Proposition~\ref{p:n=1}.
      There are exactly $\varphi(m)$ extremal solutions of width $1$ (where
      $\varphi$ is the Euler function).

   \item[Example 2] Here is a family of extremal solutions $x$ with $\sigma(x) = 2$:
      Let $m \geq 3$, and consider $x = (m-2)\,e_i + e_j$ where $i$ is coprime
      with $m$ and \mbox{$j \equiv 2i \pmod m$}. There are $\varphi(m)$
      extremal solutions of this type, and we'll see that they cover all extremal
      solutions of width $2$.
\end{description}

\section{Main Results}\label{s:res}

In this section we state the results. The proofs are postponed to the following sections.

Example~1 and Example~2 essentially cover all extremal solutions:

\begin{theorem}\label{th:extr}
   Assume $m \geq 3$, $m \neq 6$. Then all extremal solutions of ${\bf (C}_m{\bf )}$
   have widths $\sigma(x) = 1$ or $2$. There are exactly $2\,\varphi(m)$ extremal solutions.
\end{theorem}

For $m = 6$ there are exactly two additional extremal solutions:
\mbox{$2\,e_2 + e_3 + e_5 = (0,2,1,0,1)$} and $e_1 + e_3 + 2\,e_4 = (1,0,1,2,0)$,
thus the number of extremal solutions is $2\, \varphi(6) + 2 = 6$.

From Theorem~\ref{th:extr} we derive a somewhat stronger version of the Eggleton-Erd\H{o}s bound:

\begin{kor}\label{c:sig3}
   Let $m \geq 3$, $m \neq 6$, and $x$ be an indecomposable solution of ${\bf (C}_m{\bf )}$
   of width $\sigma(x) \geq 3$. Then the total size is $\|x\|_1 + \sigma(x) \leq m$
   and the length is \mbox{$\|x\|_1 \leq m-3$}.
\end{kor}

\begin{kor}\label{c:ht}
   Let $m \geq 3$, $m \neq 6$, and $x$ be an extremal solution of ${\bf (C}_m{\bf )}$.
   Then the height is $\|x\|_{\infty} \geq m-2$.
\end{kor}

As an additional result we provide two upper bounds for the number of indecomposable
solutions.

\begin{theorem}\label{th:ell2}
   For $m \geq 4$ the number $\ell(m)$ of indecomposable solutions of $({\bf C}_m{\bf )}$
   satisfies
\[
    \ell(m) \leq \sum_{s=1}^{\text{\sf SD}(m)} {m-1 \choose {s, s-1, m-2s}},
\]
   a sum of trinomial coefficients, in particular $\ell(m) < 3^{m-1}$ for $m \geq 2$.
\end{theorem}

\subsection*{Application to zero-sum theory}

We translate the results into the language of zero-sum multisets. A multiset $S$ consists
of a supporting set $\supp(S)$ and an integer-valued function that assigns a multiplicity
$\mu(a) \in \N$ to each element $a \in \supp(S)$. If $\supp(S)$ is a finite subset of
a $\Z$-module (or additively written abelian group) $M$, then the multiset sum of
$S$ is
\[
     \Sigma(S) = x_1 a_1 + \cdots + x_n a_n \in M
\]
where $\supp(S) = \{a_1, \ldots, a_n\}$ and $x_i = \mu(a_i)$. The multiset
$S$ is called a zero-sum multiset if $\Sigma(S) = 0$, and it is minimal if
no proper nonempty submultiset has a zero sum. The size of $S$ is
\[
     \#S := \sum_{a \in \supp(S)} \mu(a) = \|x\|_1
\]
(the number of its elements counted according to their multiplicities),
the width of $S$ is $\sigma(S) = \#\supp(S) = \sigma(x)$ (the number of
different elements). The Eggleton-Erd\H{o}s bound is
\[
     \#S + \sigma(S) \leq m+1
\]
for a minimal zero-sum multiset $S$ in $\Z/m\Z$. In this context Theorem~\ref{th:extr}
and Corollary~\ref{c:sig3} read as follows:

\begin{kor}\label{c:zsms}
   Let $S$ be a minimal zero-sum multiset in $\Z/m\Z$.
\begin{proplist}
   \item If $\#S + \sigma(S) = m+1$, then $\sigma(S) \leq 2$ except when $m = 6$ and
      \mbox{$S = (1,3,4,4)$} or $S = (2,2,3,5)$.
   \item If $m \geq 3$, $m \neq 6$, and $\#S + \sigma(S) = m+1$, then $S$ contains
      an element $a$ of multiplicity $\mu(a) \geq m-2$.
   \item If $m \geq 7$ and the width is at least $\sigma(S) \geq 3$, then the size
      is bounded by $\# S \leq m - \sigma(S)$.
\end{proplist}
\end{kor}
In other words, ``broad'' (= large width) minimal zero-sum multisets are ``short''
(= small size). Or ``long'' minimal zero-sum multisets are ``narrow''. Some
different but unrelated results along these lines are in \cite{YL, ZYL}.

\subsection*{Application to invariant theory}

Let $k$ be a field that contains a primitive $m$-th root of unity, in
particular $\charac k \not| \: m$. Then each representation of the cyclic
group $G = \mathcal{Z}_m$ of order $m$ is diagonalizable: We find a basis
such that the corresponding operation on the polynomial ring
$k[X] = k[X_1,\ldots,X_n]$ is given by the formula
\[
   A \cdot X_i = \varepsilon^{a_i} X_i \quad \text{for } i = 1, \ldots, n,
\]
with suitable $a_i \in \Z$, $0 \leq a_i \leq m-1$, where $A$ is a fixed
generator of the cyclic group $G$ and $\varepsilon$, a fixed
primitive $m$-th root of unity.

A polynomial $f = \sum_{\nu \in \N^n} c_{\nu} X^{\nu}$---with the usual compact
notation \mbox{$X^{\nu} = X_1^{\nu_1} \cdots X_n^{\nu_n}$} for the
multidegrees $\nu = (\nu_1, \ldots, \nu_n)$---transforms to
\[
   A \cdot f = \sum_{\nu \in \N^n} \varepsilon^{(a|\nu)} c_{\nu} X^{\nu},
\]
where $(a|\nu) = a_1 \nu_1 + \cdots + a_n \nu_n$. Thus $f$ is invariant
if and only if it has only monomials with $(a|\nu) \equiv 0 \pmod m$. A
minimal system of generators of the invariant algebra therefore consists
exactly of the monomials $X^{\nu}$ for which $\nu$ is an indecomposable
solution of the congruence {\bf (A)}. These generators are usually called
the fundamental invariants. For a monomial $X^{\nu}$ the length
$\|\nu\|_1$ is the total degree $\deg X^{\nu}$, and the width $\sigma(\nu)$
counts the different variables $X_i$ that occur in $X^{\nu}$.

For simplicity we consider the special case $n = m-1$, $a_i = i$, the
``standard representation'' that corresponds to the standard congruence
${\bf (C}_m{\bf )}$. The Eggleton-Erd\H{o}s bound restricts the degrees of
the fundamental invariants to $\deg X^{\nu} + \sigma(\nu) \leq m+1$.
Theorem~\ref{th:extr} and Corollary~\ref{c:sig3} yield:

\begin{kor}
   Let $f = X^{\nu}$ be a fundamental invariant of the standard representation
   of the cyclic group $G$.
\begin{proplist}
   \item If $\deg X^{\nu} + \sigma(\nu) = m+1$, then $\sigma(\nu) \leq 2$
      except when $m = 6$ and $f = X_1 X_2 X_4^2$ or $f = X_2^2 X_3 X_5$.
   \item If $m \geq 3$, $m \neq 6$, and $\deg X^{\nu} + \sigma(\nu) = m+1$,
      then the monomial $f$ contains a variable $X_i$ of degree $\nu_i \geq m-2$.
   \item If $m \geq 7$ and $f$ contains $s \geq 3$ different variables, then
      its total degree is bounded by $\deg f \leq m - \sigma(\nu)$.
\end{proplist}
\end{kor}

Moreover Theorem~\ref{th:extr} describes the $2\cdot \varphi(m)$ ``extremal''
invariants.

The number $\ell(m)$ of Theorem~\ref{th:ell2} also counts the fundamental invariants. Hence
Theorem~\ref{th:ell2} provides upper bounds for the size of the system of fundamental
invariants, called the embedding dimension of the invariant algebra. This has even
an application to the classical invariant theory of binary forms ($SL_2$-invariants),
see \cite{Kac} or \cite{DEN}.

\section{Zerofree Subsets}\label{s:zfs}

Call a subset $T \subseteq \Z/m\Z$ zerofree if no sum $\Sigma(U)$,
$U \subseteq T$, $U \neq \emptyset$, is $0$ in $\Z/m\Z$. (Note that
$\Sigma(\emptyset) = 0$.) Let $\Delta(T)$ be the number of different
subset sums $\Sigma(U)$ in $\Z/m\Z$ (including $0$).

\begin{lemma}\label{l:E-E}
   Let $T \subseteq \Z/m\Z$ be zerofree with $r$ elements.
\begin{proplist}
   \item $\Delta(T) \geq 2r$.
   \item If $r \geq 4$, then $\Delta(T) \geq 2r + 1$.
\end{proplist}
\end{lemma}
\begin{proof}
   See Theorems~4 and 5 of \cite{E-E}. (Note that in \cite{E-E} the empty sum $0$
   is not counted.)
\end{proof}

\begin{description}
   \item[Note] Olson \cite{Ol75} has the stronger bound $\Delta(T) \geq r^2/9$
      for $m$ large enough. We don't make use of this bound in the proof of
      Theorem~\ref{th:extr}.
\end{description}

We need a more concrete version of Lemma~\ref{l:E-E} for the case $r = 3$
and begin with two auxiliary lemmas.

\begin{lemma}\label{l:r3a}
   Let $T = \{t_1, t_2, t_3\} \subseteq \Z/m\Z$ be
   zerofree.
\begin{proplist}
   \item The five subset sums $0$, $t_1$, $t_2$, $t_3$, $t_1 + t_2 + t_3$
      are different.
   \item Assume $\{i, j, k\} = \{1, 2, 3\}$. Then the sum $t_i+t_j$ equals some other
	    subset sum of $T$ if and only if $t_i + t_j = t_k$.
   \item $\Delta(T) = 8 - s$ where $s$ is the number of true equations in
      the system
\begin{align}
     \label{eq:3_1}   t_1 + t_2 & \stackrel{?}{=} t_3 \\
     \label{eq:3_2}   t_1 + t_3 & \stackrel{?}{=} t_2 \\
     \label{eq:3_3}   t_2 + t_3 & \stackrel{?}{=} t_1
\end{align}
\end{proplist}
\end{lemma}
\begin{proof}
   (i) is trivial.

   (ii) The subset sum $t_i + t_j$ is different from $0$, $t_i$, $t_j$, $t_i + t_k$,
	 $t_j + t_k$, and $t_i + t_j + t_k = \Sigma(T)$. The only remaining possibility
	 is $t_i + t_j = t_k$.

   (iii) By (ii) the equations (\ref{eq:3_1})--(\ref{eq:3_3}) describe the only way
   a two-element sum might equal any of the other subset sums.
\end{proof}

Lemma~\ref{l:E-E} (i) implies that at least one of (\ref{eq:3_1})--(\ref{eq:3_3})
must be false. This can also be seen directly.

For even $m$ consider the sets
\[
     T(m,a) := \{a,\: \frac{m}{2},\: \frac{m}{2} + a\}
     \quad\text{where } 1 \leq a < \frac{m}{2},\: a \neq \frac{m}{4}\,.
\]

\begin{lemma}\label{l:Tma}
   $T(M,a)$ is zerofree as a subset of $\Z/m\Z$, and $\Delta(T(m,a)) = 6$.
\end{lemma}
\begin{proof}
   The eight subset sums $\Sigma(U)$ are
\[
     0,\: a,\: \frac{m}{2}, \: \frac{m}{2} + a,\: a + \frac{m}{2},\:
     \frac{m}{2} + 2a,\: \frac{m}{2} + \frac{m}{2} + a = a,\: 2a,
\]
   six of which are different (since $a \neq m/4$).
\end{proof}

\begin{lemma}\label{l:r3b}
   Let $m \geq 6$ and $T \subseteq \Z/m\Z$ be zerofree with \mbox{$\# T = 3$}.
	 Then $\Delta(T) \geq 6$, and the following statements are equivalent:
   \begin{proplist}
      \item $\Delta(T) = 6$.
      \item $m$ is even, and $T = T(m,a)$ for some $a$ with $1 \leq a < m/2$,
			   $a \neq \frac{m}{4}$.
   \end{proplist}
\end{lemma}
\begin{proof}
   Let $T = \{t_1, t_2, t_3\}$. By Lemma~\ref{l:E-E} (i) $\Delta(T) \geq 6$.
   In the case of equality exactly two of the conditions (\ref{eq:3_1})--(\ref{eq:3_3})
   must be true. Without loss of generality we may assume (\ref{eq:3_1}) and
   (\ref{eq:3_2}) are true. Then $t_1 + t_2 = t_3 = t_2 - t_1$, hence
   $2t_1 = 0$, thus $t_1 = m/2$. Assume $t_2 < t_3$ (as integers), and let
   $a := t_2$. Then $t_3 = m/2 + a$.
\end{proof}

\section{Extremal Solutions of Width $2$}\label{s:wtwo}

We consider the congruence
\[
      {\bf (A_2)}\qquad\qquad\qquad\qquad\qquad   ax + by  \equiv  0  \pmod m.
      \qquad\qquad\qquad\qquad\qquad\qquad
\]
for two unknown integers $x, y \in \N$, assuming that $m \in \N_3$, $a, b \in \N$,
and $a \not\equiv b \pmod m$. For an indecomposable solution $(x,y)$ with $x \neq 0$,
$y \neq 0$, the width is $\sigma(x,y) = 2$, and the size is $x+y \leq m-1$. The
next lemma characterizes the extremal ones among them:

\begin{lemma}\label{l:extr2}
   Assume that $m \in \N_3$, $a, b \in \N$, and $a \not\equiv b \pmod m$.
   Let \mbox{$(x,y) \in \N^2$} be an indecomposable solution of ${\bf (A_2)}$ with
   $x \neq 0$, $y \neq 0$, and $x + y = m-1$. Then one of the following
   statements is true:
\begin{proplist}
   \item $x = m-2$, $y = 1$, $\gcd(m,a) = 1$, and if $c$ is the $\bmod\,m$-inverse
      of $a$ (i.\,e. $ca \equiv 1 \pmod m$), then $cb \equiv 2 \pmod m$.
   \item $x = 1$, $y = m-2$, $\gcd(m,b) = 1$, and if $c$ is the $\bmod\,m$-inverse of $b$,
      then $ca \equiv 2 \pmod m$.
\end{proplist}
\end{lemma}

\begin{description}
   \item[Remark] The two items (i) and (ii) describe the same set of cases, only
      with the denotations of $a$ and $b$ interchanged. The second statements for
      both cases follow directly, since (for instance)
      \mbox{$0 \equiv c\,(ax+by) \equiv m-2 + cb \pmod m$}.
   \item[Example] If $a = 1$, $b = 2$, then $(m-2,1)$ is an extremal solution.
      This is essentially the only example: Lemma~\ref{l:extr2} tells us that by the
      action of the mutliplicative group $\bmod\,m$ we get all extremal solutions for
      a fixed module $m$ and varying coefficients $a$ and $b$.
\end{description}

For the proof of Lemma~\ref{l:extr2} we use some auxiliary lemmas:

\begin{lemma}\label{l:red2}
   Let $m \in \N_2$ and $a, b \in \N$.
\begin{proplist}
   \item Let $d := \gcd(m,a)$ and $d' := \gcd(d,b)$, $d = d'e$. Then for
      $(x,y) \in \N_2$ the following two statements are equivalent:
      \begin{enumerate}
         \item $(x,y)$ is an indecomposable solution of ${\bf (A_2)}$.
         \item $x < \frac{m}{d}$, $e|y$, and $(x, \frac{y}{e})$ is an indecomposable solution
            of
\[
     \frac{a}{d}\cdot s + \frac{b}{d'}\cdot t \equiv 0 \pmod{\frac{m}{d}}.
\]
      \end{enumerate}

   \item If $a$ and $m$ are coprime, then the indecomposable solutions of ${\bf (A_2)}$
      are exactly the same as for $1\cdot x + b'\cdot y \equiv 0 \pmod m$ where
       $c$ is the inverse of $a$ modulo $m$, and $b' = bc \bmod m$.
\end{proplist}
\end{lemma}
\begin{proof}
   (i) If $x \geq \frac{m}{d}$,then $(\frac{m}{d},0)$ is a solution $\leq (x,y)$.
   If $ax + by = km$, then $d|by$, $e|\frac{b}{d'}\,y$, hence $e|y$.
   Thus
\[
     ax + by = km \Leftrightarrow \frac{a}{d}\,x + \frac{b}{d'}\,\frac{y}{e} = k\,\frac{m}{d}.
\]
   Therefore the mapping $(x,y) \mapsto (x, \frac{y}{e})$ is a bijection between
   the respective sets of solutions, and obviously it preserves the indecomposability
   (in both directions).

   (ii) We have $ac \equiv 1 \pmod m$, thus
   $ax + by = km \Leftrightarrow x + bcy = kcm$.
\end{proof}

\begin{lemma}\label{l:red3}
   Assume that $m \in \N_4$, $a, b \in \N$, $a \not\equiv b \pmod m$, and ${\bf (A_2)}$
   has an indecomposable solution $(x,y) \in \N^2$ with $x \neq 0$, $y \neq 0$, and
   $x + y = m-1$. Then at least one of $a$ and $b$ is relative prime with $m$.
\end{lemma}
\begin{proof}
   Assume that $d = \gcd(m,a) > 1$, and let $d' = \gcd(d,b)$ and $e = d/d'$.
   Then by Lemma~\ref{l:red2} (i) we have $e|y$, and $(x,y')$
   is an indecomposable solution of $a's + b't \equiv 0 \pmod{m'}$ where $y' = y/e$,
   $a' = a/d$, $b' = b/d'$, and $m' = m/d$. If $a' \not\equiv b' \pmod{m'}$, then
\[
     x + \frac{y}{e} \leq \frac{m}{d} - 1, \quad x + y \leq dx + d'y \leq m - d < m-1,
\]
   contradiction. Otherwise $a' \equiv b' \pmod{m'}$, thus $x + y' = m'$, and
\[
     x + y \leq ex + y = \frac{m}{d/e}\,.
\]
   Since $m \geq 4$ this is possible only if $e = d$ and $e = 2$, thus $d' = 1$,
   $b' = b$. Then $2b = 2b' \equiv 2a' = a \pmod m$, and $\gcd(m,2b) = \gcd(m,a) = 2$,
   hence $\gcd(m,b) = 1$.
\end{proof}

\begin{lemma}\label{l:m+b}
  Let $m \in \N_1$, $b \in \N$. The assignment $(s,t) \mapsto (s,t+u)$
  with $u = \frac{s+bt}{m}$ defines a bijection between
\begin{proplist}
  \item the set of indecomposable solutions of
    $s + bt \equiv 0 \pmod m$

  \item and the set of indecomposable solutions of
    $x + by \equiv 0 \pmod{m+b}$ except $(m+b,0)$.
\end{proplist}
\end{lemma}
\begin{proof}
   See \cite{Tin1} (or \cite[Prop. 1]{KP_LC2}).
\end{proof}

\begin{lemma}\label{l:a=1}
   Let $m \geq 3$, $a = 1$, $2 \leq b \leq m-1$, and $(x,y)$ an indecomposable solution
	 of $ax + by \equiv 0 \pmod m$, where $x \neq 0$, $y \neq 0$, $x+y = m-1$.
	 Then
\begin{proplist}
   \item $b = 2$, $x = m-2$ and $y = 1$,
   \item or $m=5$, $b=3$, $x=1$, $y=3$.
\end{proplist}
\end{lemma}
\begin{proof}
   We prove Lemma~\ref{l:a=1} by induction on $m$ using Lemma~\ref{l:m+b}.
	 If $m = 3$, then $b = 2$. All indecomposable solutions are $(3,0)$, $(1,1)$, $(0,3)$.
	 Hence $x = 1 = m-2$ and $y = 1$.

   Now we assume that $m \geq 4$. Since $(x,y)$ is not the solution $(m,0)$, by
	 Lemma~\ref{l:m+b} it has the form $(s, t+u)$ where $(s,t)$ is an indecomposable
	 solution of $s + bt \equiv 0 \pmod{m-b}$ with $s = x \neq 0$,
	 $s + t = x + y - u = m - 1 - u$,
	 and
\[
     u = \frac{s+bt}{m-b}.
\]

   {\bf Case I,} $t = 0$. Then $(s,t) = (m-b,0)$, $u = 1$, $(x,y) = (m-b, 1)$,
	 $m-1 = x + y = m - b + 1$, $b = 2$, $(x,y) = (m-2, 1)$, as asserted.

   {\bf Case II,} $t \geq 1$. Then $s < m-b$, $u < [(m-b) + b\,(m-b)]/(m-b) = 1 + b$,
	 hence $u \leq b$, and $s + t = m - 1 - u \geq m - b - 1$.

   {\bf Case IIa,} $b \equiv 0 \pmod{m-b}$. Since $s = x \neq 0$, the solution $(s,t)$
	 must be equal to $(m-b, 0)$, hence we are back in case I.

   {\bf Case IIb,} $b \equiv 1 \pmod{m-b}$. Then $(s,t)$ is an indecomposable solution
	 of $s + t \equiv 0 \pmod{m-b}$, hence $s + t = m-b$. From
\[
     m - 1 = x + y = s + t + u = m - b + u
\]
   we conclude that $u = b-1$. Using $s = m-b-t$ we conclude that
\[
     u = \frac{s + bt}{m-b} = \frac{m-b-t + bt}{m-b} = 1 + \frac{(b-1)\,t}{m-b}
       = 1 + \frac{ut}{m-b}\,,
\]
\[
     1 = u - \frac{ut}{m-b} = \frac{(m-b-t)\,u}{m-b} = \frac{su}{m-b}\,.
\]
   Therefore $s = m-b-t\:|\:m-b$, and $t \:|\:m-b$. Since $s + t = m-b$, and
   both $s,t \geq 1$, this is possible only if
\[
     s = t = \frac{m-b}{2}\,.
\]
   Then $u = (m-b)/s = 2$, $b = u+1 = 3$, and $b \equiv 1 \pmod{m-b}$ enforces
   one of
\begin{itemize}
   \item $m-b = 1$, contradicting $2 \:|\:m-b$, or
   \item $m-b = 2$, $m = 5$, $s = t = 1$, $x  = 1$, $y = 3$.
\end{itemize}

   {\bf Case IIc,} $b \not\equiv 0, 1 \pmod{m-b}$. Then $s + t \leq m - b - 1$, hence
	 $= m - b - 1$, and $u = b$. In this case we may apply the induction hypothesis and
	 get $s = m-b-2$, $t = 1$, $b = u = (m-2)/(m-b)$, $b\,(m-b) = m-2$,
	 $(b-1)\,m = b^2 -2$,
\[
     2 \equiv b^2 = (b-1+1)^2 = (b-1)^2 + 2\,(b-1) + 1 \equiv 1 \pmod{b-1},
\]
   hence $b-1 = 1$, $b = 2$, $m = b^2 - 2 = 2$, contradiction.
\end{proof}

For the proof of Lemma~\ref{l:extr2} we may assume without loss of generality
that $a, b \in \{0,\ldots,m-1\}$ and, by Lemma~\ref{l:red3}, that $a$ is
relative prime with $m$, by Lemma~\ref{l:red2} (ii), that $a = 1$.
Then we are in the situation of Lemma~\ref{l:a=1}, and thus the proof of
Lemma~\ref{l:extr2} is complete.

\setcounter{kor}{0}
\begin{kor}
   The congruence ${\bf (A_2)}$ admits an extremal solution if and only if $a$ is
   coprime with $m$ and $b \equiv 2a \pmod m$, or $b$ is coprime with $m$ and
   $a \equiv 2b \pmod m$. This extremal solution, $(m-2,1)$ or $(1,m-2)$, is unique.
\end{kor}

\begin{kor}
   If $x$ is an extremal solution of ${\bf (C}_m{\bf )}$ and $\sigma(x) = 2$
   (hence \mbox{$m \geq 3$}), then \mbox{$x = (m-2)\,e_i + e_j$}
   where $i$ is coprime with $m$ and \mbox{$j \equiv 2i \pmod m$}.
   There are exactly $\varphi(m)$ extremal solutions of width two.
\end{kor}

\section{Proof of Theorem~\ref{th:extr}}\label{s:proof}

\begin{lemma}\label{l:extrs}
   Let $m \geq 4$ and $x$ be an extremal solution of ${\bf (C}_m{\bf )}$. Then
\begin{proplist}
   \item there is exactly one index $j$ with $x_j \geq 2$,
   \item $\sigma(x) \leq 3$.
\end{proplist}
\end{lemma}
\begin{proof}
   Let $s = \sigma(x)$. Then $2s \leq m$ by the remark in
   Section~\ref{s:indsol}. Let
\[
                      y  =  \sum_{i\in\supp(x)} e_i.
\]
  If $x = y$, then there is no coordinate $\geq 2$, hence all $x_i \leq 1$.
  Then \mbox{$\|x\|_1 = \sigma(x) = s$}, hence $2s = m+1$, contradiction.

  Otherwise, since the $2^s$ weights $\alpha(u)$ for $0 \leq u \leq y$ are
	subset sums of $T = \supp(x)$, by Lemma~\ref{l:E-E} (i) they represent
  at least $2s$ different residue classes $\bmod\,m$. In each chain
\[
   0 < u^{(1)} < \ldots < u^{(s)} = y < u^{(s+1)} < \ldots < u^{(m-s+1)} = x
\]
  there remain only $m-2s$ possible values $\alpha(u^{(j)})$ for the $m-2s$
  indices $j$ with $s+1 \leq j < m-s+1$.
  So if we exchange a single element of the chain between $y$ and $x$,
  the weights of the old and of the new element must coincide.
  Now assume that $x_i \geq 2$ and $x_j \geq 2$ with $i \neq j$. Then
  $y + e_i + e_j \leq x$, and for the intermediate step between $y$ and
  $y + e_i + e_j$ we have the
  two choices $y + e_i$ and $y + e_j$. Hence $\alpha(y+e_i) \equiv \alpha(y+e_j)$.
  This implies $i = \alpha(e_i) \equiv \alpha(e_j) = j$, whence $i = j$,
  and (i) is proved.

  Moreover for $s \geq 4$ the values $\alpha(u)$ in the previous paragraph represent
  at least $2s + 1$ different residue classes by Lemma~\ref{l:E-E} (ii), leaving not
  enough room for the weights of the vectors between $y$ and $x$, contradiction.
  This proves (ii).
\end{proof}

We now prove Theorem~\ref{th:extr}.

For an extremal solution $x$ we denote the one coordinate $x_j \geq 2$ by $u$,
all other coordinates are $x_i \leq 1$. Multiplying the congruence ${\bf (C}_m{\bf )}$
by a number that is relatively prime with $m$ (and reducing the coefficients $\bmod\,m$)
doesn't change the solutions (up to a permutation of the indices $1, \ldots, m-1$)
nor their widths or lengths. Therefore we may assume that $j = d \,|\,m$,
see \cite{KP_OMGr}. In this situation ${\bf (C}_m{\bf )}$ has the form
\[
     d\cdot u + \Sigma(S) \equiv 0 \pmod m
\]
where $S \subseteq \{1, \ldots, m-1\} - \{d\}$ and $\Sigma(S) = \sum_{i\in S} i$
is the sum of the elements of $S$. Thus
\[
     x = u\cdot e_d + \sum_{i\in S} e_i.
\]
Let $s := \# S$ be the size of $S$, so $\sigma(x) = s+1$. Since the cases $\sigma(x) = 2$
and $\sigma(x) \geq 4$ are settled by Lemmas~\ref{l:extr2} and \ref{l:extrs} we may
assume that $\sigma(x) = 3$, or $s = 2$. Since $\|x\|_1 = u + s$ and $\sigma(x) = 1 + s$,
the extremality condition translates to the equation \mbox{$m + 1 = u + 2 + 1 + 2$}, or
\[
     u + 4 = m
\]
(and thus $m \geq 6$). We have $u < m' := m/d$ for otherwise the solution
$m'\, e_d < x$ contradicts the minimality of $x$. In particular
\[
     du < m.
\]
(By the way this implies that $d \leq 2$.)
We may shrink the potential range of $S$ due to the observation
\[
     m - wd \not\in S \quad\text{for } w = 1 \ldots, u,
\]
for otherwise $ m -wd \in S$ makes $w\,e_d  + e_{m-wd}$ a solution that is $< x$ except
in the case $w = u$ and $m - ud = d$---but then also $m - wd = d \not\in S$.

Now we consider the set
\[
     R := \{0, \ldots, m-1\} - \{m - wd\:|\: 1 \leq w \leq u\}
\]
with $S \subseteq R - \{0, d\}$ (note that maybe $d \in R$ and that the
removed elements $m - wd$ are multiples of $d$). Its size is $\# R = m - u = 4$.
Let $T := S \cup \{d\}$ (that is, $T = \supp(x)$). Then $r := \# T = s+1 = 3$,
$2r \leq m$. If we let $U$ run through all the $8$ subsets of $T$, then by
Lemma~\ref{l:r3b} one of the following two statements must be true:

\begin{enumerate}
   \item $T$ is not zerofree, there is a subset $U \subseteq T$, $U \neq \emptyset$,
      with $m\,|\,\Sigma(U)$.
   \item The sums $\Sigma(U)$ represent $\geq 6$ different residue classes $\bmod\,m$,
      even $\geq 7$ different classes, except when $T$ is one of the exceptional sets
      $T(m,a)$.
\end{enumerate}
Statement 1 makes $\sum_{i \in U} e_i$ a solution of ${\bf (C}_m{\bf )}$
that is $\leq e_d + \sum_{i \in S} e_i < x$, contradiction.
Hence statement 2 is true.

{\bf Case I,} the $\Sigma(U)$ represent at least seven classes, thus at least three
outside of $R$. Then at least two have the form $\Sigma(U) \equiv m-wd \pmod m$
with $1 \leq w < u$. If $d \not\in U$, then $U \subseteq S$, and
$w\,e_d + \sum_{i \in U} e_i$ is a solution $< u\,e_d + \sum_{i \in S} e_i = x$
of ${\bf (C}_m{\bf )}$, contradiction.

If however $d \in U$, then
\[
     y = w\,e_d + \sum_{i \in U} e_i = (w+1)\,e_d + \sum_{i \in U-\{d\}} e_i
\]
is a solution $\leq x$. The minimality of $x$ enforces $w\,e_d + \sum_{i \in U} e_i = x$,
that is $w = u-1$, and $U = S \cup \{d\}$. But there is yet another residue class
outside of $R$ of the form $\Sigma(V) \equiv m - vd$ with $1 \leq v < u$,
$v \neq w = u-1$, hence $v \leq u-2$. Thus $v\,e_d + \sum_{i \in V} e_i$ is a
solution $< x$, contradiction.

{\bf Case II,} $T = \{a, m/2, a + m/2\}$ with $1 \leq a < m/2$
and $a \neq m/4$. In particular $m$ is even and $d = a$:
\begin{quote}
   We know that $d \in T$. Since $d \leq m/2$, $d$ cannot be $a + m/2$.
   The assumption $d = m/2$ implies
\[
     x = e_a + (m-4)\,e_{m/2} + e_{a + m/2},
\]
\[
     0 \equiv \alpha(x) = a + \frac{m}{2}\cdot(m-4) + a + \frac{m}{2} \equiv 2a + \frac{m}{2}.
\]
   Since $0 < a < m/2$, this implies $2a = m/2$, contradicting \mbox{$a \neq m/4$}.
\end{quote}
Since $T = S \cup \{d\}$ we conclude that $S = \{m/2, a+m/2\}$ and
\[
     x = (m-4)\,e_a + e_{m/2} + e_{a+m/2},
\]
\[
     0 \equiv \alpha(x) = a\cdot (m-4) + \frac{m}{2} + a + \frac{m}{2} \equiv -3a,
\]
hence $3a = m$, $d = a = m/3$, and $m$ is a multiple of $6$, say $m = 6n$. Then $u = 6n - 4$,  
$a = 2n$, $S = \{3n, 5n\}$,
\[
     x = (6n-4)\,e_{2n} + e_{3n} + e_{5n}.
\]
Since $2\cdot 2n + 3n + 5n = 12 n$ the vector $2\,e_{2n} + e_{3n} + e_{5n}$ is
a solution $\leq x$, hence $= x$, $6n - 4 = 2$, $n = 1$.

This finishes the proof of Theorem~\ref{th:extr} and by the way identifies the
exceptions for $m = 6$.

\section{Alternative Proof of Theorem~\ref{th:extr}}\label{s:proof}

We give an alternative proof that is much shorter but uses the deep results of
\cite{SaCh} and \cite{Yuan} on Elashvili's conjecture:

\begin{description}
   \item[ESCY Theorem] (Elashvili, Savchev/Chen, Yuan) If $x$ is an indecomposable
      solution of ${\bf (C}_m{\bf )}$ of length \mbox{$\|x\|_1 \geq \lfloor m/2 \rfloor + 2$},
      then the index of $x$ is $1$.

   \item[Remark] The multiplicative group $G = (\Z/m\Z)$ of order $\varphi(m)$ acts
      in a natural way by permuting the indices $1, \ldots, m-1$, hence permutes the
      solutions of ${\bf (C}_m{\bf )}$ as well as the indecomposable solutions.
      The weight $\alpha(x)$ of a solution is a multiplie of $m$, and the index
      is the minimum of $\alpha(u)/m$ where $u$ ranges over the $G$-orbit of $x$.
\end{description}

\begin{lemma}\label{l:lev1}
   Let $m \geq 3$ and $x$ be an extremal solution of ${\bf (C}_m{\bf )}$ of weight
   $\alpha(x) = m$. Then $x = m e_1$ or $x = (m-2)\,e_1 + e_2$.
\end{lemma}
\begin{proof}
   Let $s = \sigma(x)$. If $s = 1$, then we have $\|x\|_1 = m$, $x = m e_i$,
   \mbox{$m = \alpha(x) = mi$}, hence $i = 1$, $x = m e_1$.

   Now assume $s \geq 2$ and
\[
     \supp(x) = \{i_1, \ldots, i_s\} \quad\text{with } 1 \leq i_1 \leq \ldots \leq i_s \leq m-1.
\]
   In particular $i_{\nu} \geq \nu$ for $\nu = 1, \ldots, s$. Extremality means
\[
     \sum_{\nu=1}^s x_{i_{\nu}} = \|x\|_1 = m + 1 - s.
\]
   From the chain
\begin{align*}
     m & = \alpha(x) = \sum_{\nu=1}^s i_{\nu} x_{i_{\nu}} \geq \sum_{\nu=1}^s \nu x_{i_{\nu}}
         = \sum_{\nu=1}^s x_{i_{\nu}} + \sum_{\nu=1}^s (\nu-1)\,x_{i_{\nu}} \\
       & = m - (s-1) + \sum_{\nu=2}^s (\nu-1)\,x_{i_{\nu}} \geq m - (s-1) + (s-1) = m
\end{align*}
   of equalities and inequalities we conclude that
\[
     \sum_{\nu=2}^s (\nu-1)\,x_{i_{\nu}} =  s - 1,
\]
   which is possible only if $s = 2$ and $x_{i_2} = 1$. Set $i_1 = i$ and $i_2 = j$.
   Since $x_j = 1$ and $m-1 = \|x\|_1 = x_i + x_j$ we have $x_i = m-2$, thus
   $x = (m-2)\,e_i + e_j$ and $\alpha(x) = i\cdot(m-2) + j$. The case $m = 3$ being settled
   we may assume that $m \geq 4$. Then necessarily $i = 1$ and consequently $j = 2$.
\end{proof}

\setcounter{kor}{0}
\begin{kor}
   If $m \geq 3$ and $x$ is an extremal solution of ${\bf (C}_m{\bf )}$ of index one,
   then $x$ has one of the forms
\begin{proplist}
   \item $x = m e_i$ where $i$ is coprime with $m$,
   \item $x = (m-2)\,e_i + e_j$ where $i$ is coprime with $m$ and $j = 2i \bmod m$.
\end{proplist}
\end{kor}
\begin{proof}
   These are the elements in the $G$-orbits of $x = m e_1$ and \mbox{$x = (m-2)\,e_1 + e_2$}.
\end{proof}

From this result we derive the alternative proof of Theorem~\ref{th:extr}:

Let $x$ be an extremal solution of ${\bf (C}_m{\bf )}$, and $s = \sigma(x)$. Then
the length of $x$ is $\|x\|_1 = m + 1 - s$, and
\begin{equation}\label{eq:mhalf}
     m+1-s \geq \left\lfloor \frac{m}{2} \right\rfloor + 2 \quad\Longleftrightarrow\quad
     s \leq m - \left\lfloor \frac{m}{2} \right\rfloor - 1 = \left\lceil \frac{m}{2} \right\rceil - 1.
\end{equation}
If $m$ is odd, then $\lceil m/2 \rceil - 1 = \lfloor m/2 \rfloor$, hence (except for
the trivial case $m = 3$) the condition in (\ref{eq:mhalf}) is satisfied by the
Remark in Section~\ref{s:indsol}. The ESCY Theorem applies and settles
Theorem~\ref{th:extr} for this case.

If $m$ is even, then $\lceil m/2 \rceil - 1 = m/2 - 1$, and by the same reasoning we are
done except in the case $s = m/2$. In this case $\|x\|_1 = 1 + m/2 = \sigma(x) + 1$, and
$x$ has one coordinate $x_i = 2$, all other coordinates $x_j = 1$ or $0$ (for $j \neq i$).
Lemma~\ref{l:extrs} (ii) implies that $s \leq 3$, and we are done except when $s = 3$,
thus $m = 6$.

The alternative proof of Theorem~\ref{th:extr} is complete.

\section{An Upper Bound for the Number of Indecomposable Solutions}\label{uppbnd}

Let $\ell(m)$ be the number of indecomposable solutions of the standard linear
congruence ${\bf (C}_m{\bf )}$. Since \cite{DEN} gives a lower bound we'll look for an
upper bound only.

To apply Theorem~\ref{th:extr} and its Corollary~\ref{c:sig3} we assume that
\mbox{$m \geq 4$}. Then the support of an indecomposable solution
has at most ${\text{\sf SD}(m)}$ elements. For each
$s \in \{1,\ldots,{\text{\sf SD}(m)}\}$ we have exactly ${m-1 \choose s}$
choices for an $s$-element subset $S = \{i_1,\ldots,i_s\} \subseteq \{1,\ldots,m-1\}$
that serves as support.

Proposition~\ref{p:n=1} says that the number of indecomposable solutions of
width $s = 1$ is
\[
     m-1 = \frac{(m-1)!}{1!\cdot 0!\cdot (m-2)!}
         = {m-1 \choose {1, 0, m-2}}.
\]

For $s = 2$ we have ${m-1 \choose 2} = (m-1)(m-2)/2$ choices for $S$. Let
$S = \{i, j\}$ with $1 \leq i < j \leq m-1$. The number of indecomposable solutions
with support in $\{i, j\}$ is $\leq m-1$, see \cite{Tin1, KP_LC2}. This number includes
the two solutions with one-element support $\{i\}$ or $\{j\}$. Thus the number
of indecomposable solutions with support $\{i, j\}$ is $\leq m-3$.
Therefore the number of indecomposable solutions of width $s = 2$ is
\[
     \leq \frac{(m-3)(m-1)(m-2)}{2} = \frac{(m-1)!}{2!\cdot 1!\cdot (m-4)!}
         = {m-1 \choose {2, 1, m-4}}.
\]

For $s \geq 3$ every indecomposable solution $x$ with support $S$ has
\mbox{$\|x\|_1 \leq m-s$} by Corollary~\ref{c:sig3} of Theorem~\ref{th:extr},
except when $x$ is one of the two exceptional solutions with $\sigma(x) = 3$
for $m=6$. We catch all the other ones by choosing arbitrary
$y_1, \ldots, y_{s-1} \geq 0$ with \mbox{$y_1+\cdots+y_{s-1} \leq m-2s$},
defining \mbox{$x_{i_{\nu}} = y_{\nu}+1$}, and choosing $x_{i_s}$ appropriately,
that is, minimal such that $m\,|\,\alpha(x)$. The number of such choices is
\mbox{${m-2s+s-1 \choose s-1} = {m-s-1 \choose s-1}$}. This proves
(for $m \geq 7$):

\begin{lemma}\label{l:nsupp}
   Let $m \geq 6$ and $s \geq 3$. Let $S \subseteq \{1,\ldots,m-1\}$ be an $s$-element
   subset. Then $S$ supports at most ${m-s-1 \choose s-1}$ indecomposable solutions
   of ${\bf (C}_m{\bf )}$.
\end{lemma}

For $m = 6$ and $s = 3$ we have the two exceptional solutions \mbox{$x = (0,2,1,0,1)$}
and $(1,0,1,2,0)$ with supports $S = \{2,3,5\}$ and $\{1,3,4\}$. These two sets
don't support any other indecomposable solutions. Since ${m-s-1 \choose s-1} = {2 \choose 2} = 1$,
Lemma~\ref{l:nsupp} is true also for $m = 6$.

Now we are done with Theorem~\ref{th:ell2}: For $s \geq 3$ the form of the summand
follows from Lemma~\ref{l:nsupp}.

The upper bound $3^{m-1}$ is a standard result on trinomial coefficients
(and trivial for $m = 2, 3$).

\begin{table}[htbp]
\caption{Comparing $\ell(m)$ with bounds and possible bounds}\label{t:mpm}
\begin{center}
\begin{tabular}{|c|rrrrrrr|} \hline
    $m$         &  4  &   5 &   6  &   7 &   8 &  9  &   10  \\ \hline
   $\ell(m)$    &  6  &  14 &  19  &  47 &  64 & 118 &  165  \\
  $m\cdot P(m)$ & 20  &  35 &  66  & 105 & 176 & 270 &  420  \\
    $q(m)$      &  6  &  16 &  45  & 126 & 357 &1016 & 2781  \\ \hline
  \hline
    $m$         &  11 &  12 & 13  &  14  &  15 &  16 &  17 \\ \hline
   $\ell(m)$    & 347 & 366 & 826 & 973  &1493 &2134 &3912 \\
  $m\cdot P(m)$ & 616 & 924 &1313 &1890  &2640 &3696 &5049 \\
    $q(m)$      &8350 &23606&64032&163891&393498&1517895&[\ldots] \\ \hline
  \hline
    $m$         &  18  &  19 &  20 &   21&  22 &  23  &  24 \\ \hline
   $\ell(m)$    &4037  &7935 &8246 &12966&17475&29161 &28064 \\
  $m\cdot P(m)$ &6930  &9310 &12540&16632&22044&28865 &37800 \\ \hline
\end{tabular}
\end{center}
\end{table}

Table~\ref{t:mpm} shows some explicit values where $P$ is the partition function, and
$q(m)$ is the bound from Theorem~\ref{th:ell2}, using the known values of ${\text{\sf SD}(m)}$.
The explicit values of $\ell(m)$ are taken from the On-line Encyclopedia of Integer
Sequences \cite{OEIS}. Figure~\ref{fig:logell} provides an illustration of these values
(extended to $m = 39$). The yellow line represents the lower bound from \cite{DEN} where
the unspecified proportionality factor is set to $1$.

\begin{figure}[htbp]
\begin{center}
\includegraphics[scale=0.65]{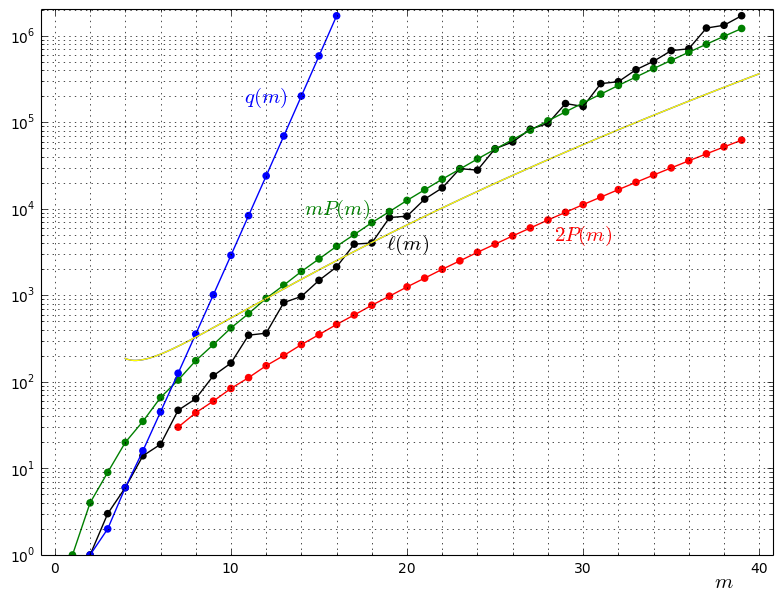}
\end{center}
\caption{The number of indecomposable solutions (semi-logarithmic scale)}\label{fig:logell}
\end{figure}

\begin{description}
   \item[Discussion] The bound $q(m)$ grows much too fast. Although significantly
      smaller than $3^{m-1}$ it seems to grow strictly exponentially. This phenomen
      has a simple heuristic explanation: In the proof of Theorem~\ref{th:ell2} we
      essentially counted all solutions in the respective simplices, not only the
      indecomposable ones. Since the solutions form the kernel of a homomorphism
      onto $\Z/m\Z$ we expect a fraction of $1/m$ of all vectors in this domain to
      yield solutions. Hence the upper bound of (volume of simplex) $\times 1/m$
      which is exponential.

      Thus for improvements we should not bother with the sum in Theorem~\ref{th:ell2} but
      rather analyze the number ${m-s-1 \choose s-1}$ in Lemma~\ref{l:nsupp} that
      overestimates the number of {\em indecomposable} solutions.

      On the other hand the value $m\,P(m)$ seems to provide a rather narrow lower
      bound for $m > 30$. This phenomen also has a heuristic explanation: The
      partitions of $m$ yield (roughly) $P(m)$ indecomposable solutions. The
      multiplicative group of order $\varphi(m)$ acts on the set of indecomposable
      solutions, and most of its orbits have size $\varphi(m)$. This
      consideration (if properly fleshed out) yields roughly $m\,P(m)$ different
      indecomposable solutions.

   \item[Some questions:]
      \begin{enumerate}
         \item Is $\ell(m) \geq m\,P(m)$ for $m > 30$\,?
         \item Is $\ell(m) \leq a\cdot e^{b\cdot \sqrt{m}}$ for certain constants
            $a$ und $b$\,?
         \item Is $\ell(m) \leq cm\cdot P(m)$ for $m \geq 2$ for some constant $c$\,?
            Note that this would imply a positive answer to question 1. Necessarily
            $c > 1$ if it exists at all since \mbox{$\ell(23) > 23\cdot P(23)$}.
         \item Or is at least $\ell(m) \leq f(m) \cdot P(m)$ for some polynomial $f$?
      \end{enumerate}
\end{description}

\end{document}